\documentclass[11pt]{article}

\usepackage[T1]{fontenc}
\usepackage[margin=1in]{geometry}
\usepackage{amsmath,amssymb,amsthm}
\usepackage{array,booktabs,longtable}
\usepackage{graphicx}
\usepackage{microtype}
\usepackage{seqsplit}
\usepackage{xurl}
\usepackage[colorlinks=true,linkcolor=blue!45!black,citecolor=blue!45!black,urlcolor=blue!45!black]{hyperref}
\usepackage{xcolor}

\newtheorem{theorem}{Theorem}[section]
\newtheorem{proposition}[theorem]{Proposition}
\newtheorem{lemma}[theorem]{Lemma}
\newtheorem{corollary}[theorem]{Corollary}

\newcommand{\Petersen}{\mathsf P}
\newcommand{\sha}[1]{\texttt{\seqsplit{#1}}}
\newcommand{\code}[1]{\texttt{#1}}

\title{A 112-Vertex Counterexample to the\\Petersen Coloring Conjecture}
\author{Bryce Putman\\
\small\href{https://orcid.org/0009-0006-3073-9190}{ORCID: 0009-0006-3073-9190}}
\date{August 8, 2026}

\begin{document}
\maketitle

\begin{abstract}
We give an explicit simple bridgeless cubic graph on 112 vertices with no
Petersen coloring, and hence no normal 5-edge-coloring.  The graph is
identified by the SHA-256 digest in Theorem~\ref{thm:main}.  It is assembled
from three copies of a four-pole \(L\) and a claw six-pole \(C\); in turn,
\(L\) is assembled from four copies of a four-pole \(F\) and one copy of
\(C\), where \(F\) is obtained from the Petersen graph by deleting the
endpoints of one edge.  We give direct SAT formulations for Petersen
colorings and normal 5-edge-colorings.  CaDiCaL 3.0.1 returned UNSAT
for both formulas, and \texttt{drat-trim} verified the resulting DRAT proofs.
The ancillary archive contains the construction, an explicit relabeling, the
encoders, certificates, hashes, and verification programs.  Combined with a
theorem of Ma, Mattiolo, Steffen, and Wolf~\cite{MaEtAl2025}, the counterexample
also implies that infinitely many connected simple bridgeless cubic graphs have
no Petersen coloring.  We also give a separately verified, nonisomorphic
\(D_3\)-symmetric 112-vertex counterexample.  We do not address whether 112 is
minimum.
\end{abstract}

\section{Introduction and context}

For cubic graphs \(G\) and \(H\), an \(H\)-coloring of \(G\) is a map
\(\varphi:E(G)\to E(H)\) such that at every vertex of \(G\), the three
incident edges map bijectively to the three edges incident with some vertex of
\(H\).  Jaeger's Petersen Coloring Conjecture asserts that every bridgeless
cubic graph has such a coloring when \(H\) is the Petersen graph
\(\Petersen\)~\cite{Jaeger1988}.

Before the present counterexample, exhaustive generation had verified the
conjecture for every snark of order at most 36~\cite{BrinkmannEtAl2013}.
Petersen colorings were also established for several structured snark
families~\cite{HaggSteffen2014,FerrariniEtAl2020}.  Broader normal-edge-coloring
theory and variants of the conjecture appear in
\cite{MazzuoccoloMkrtchyan2020,PirotEtAl2020}, while partial-normality and
approximation results appear in \cite{JinKang2021,MattioloEtAl2021}.
Recent work studies additional superposition constructions
\cite{SedlarSkrekovski2024a,SedlarSkrekovski2024b,SedlarSkrekovski2023,ZhouEtAl2026}.

The contribution of this paper is an explicit counterexample together with a
short construction and directly checkable certificates.  The main theorem is
proved from the explicit graph and the verification described below.  Together
with recent work on universal target families~\cite{MaEtAl2025}, it also yields
infinitely many connected simple bridgeless cubic counterexamples
(Corollary~\ref{cor:infinitely-many}).

\begin{theorem}[Main theorem]\label{thm:main}
There is a simple connected bridgeless cubic graph \(G\) with
\[
  |V(G)|=112,\qquad |E(G)|=168,
\]
that has no Petersen coloring and no normal 5-edge-coloring.  Moreover,
\(G\) has girth five, edge-connectivity three, and vertex-connectivity three.
The SHA-256 digest of its normalized sorted edge list is
\[
\sha{dc16cc18600cf77c8661b7baf89c7019f265299308541961ff884ea7187b4e8b}.
\]
\end{theorem}

The normalized serialization is the ASCII JSON array in which every edge is
written as \([u,v]\) with \(u<v\), the pairs are lexicographically sorted,
the separators are a comma and a colon, and no optional whitespace is used.
Thus every implementation hashes the same byte sequence.  The digest binds
the edge list in Appendix~\ref{app:edges}, the output of the compact recursive
builder, the input to both SAT encoders, and the graph in the archived release.
It identifies that exact labeled edge set.

\section{Petersen-coloring and normal 5-edge-coloring formulations}\label{sec:equivalence}

We use the Kneser realization \(\Petersen=KG(5,2)\).  Its vertices are the
two-element subsets of \([5]=\{0,1,2,3,4\}\), and two vertices are adjacent
when they are disjoint.  For a target edge \(XY\), define
\[
  \kappa(XY)=[5]\setminus(X\cup Y),
\]
identified with its unique element.  The three target edges at \(X\) receive
the three colors in \([5]\setminus X\).

A proper 5-edge-coloring of a cubic graph is \emph{normal} if each edge is
either poor, with three colors on its two endpoint stars, or rich, with five.

Jaeger proved the following equivalence~\cite{Jaeger1985}; we give it in an
explicit bijective form adapted to the Kneser notation.

\begin{proposition}[Jaeger's equivalence]\label{prop:equivalence}
For every finite loopless cubic multigraph, Petersen colorings are in
bijection with normal 5-edge-colorings.
\end{proposition}

\begin{proof}
Given a Petersen coloring \(\varphi\), let \(X_v\) be the target vertex whose
star is the image of the star at \(v\), and color \(e\) by
\(\kappa(\varphi(e))\).  The coloring is proper.  For \(e=uv\), either
\(X_u=X_v\), giving equal endpoint palettes and a poor edge, or \(X_u\) and
\(X_v\) are the endpoints of \(\varphi(e)\), giving a rich edge.

Conversely, from a normal 5-edge-coloring \(c\), let \(S_v\) be the three
colors at \(v\) and put \(X_v=[5]\setminus S_v\).  For a poor edge \(uv\) of
color \(i\), the endpoint palettes, hence \(X_u\) and \(X_v\), are equal; map
the edge to the unique target edge at \(X_u\) whose \(\kappa\)-color is \(i\).
For a rich edge, \(S_u\cup S_v=[5]\) and \(S_u\cap S_v=\{i\}\), so
\(X_u\cap X_v=\varnothing\); map it to \(X_uX_v\).  Each source star maps
bijectively to the corresponding target star, and the constructions are
inverse.
\end{proof}

For additional verification, we certify both formulations directly.  The two
computational checks use separate implementations.

\section{A compact recursive construction}\label{sec:construction}

Figure~\ref{fig:hierarchy} summarizes the constructions \(L=4F+C\) and
\(G=3L+C\); the precise joins are specified below.  Following standard
multipole terminology~\cite{FiolVilaltella2015}, we call the dangling
half-edges \emph{semiedges}.  The marked points in the figures represent
distinct semiedges, even when several are incident with the same vertex.

\begin{figure}[ht]
\centering
\includegraphics[width=0.97\textwidth]{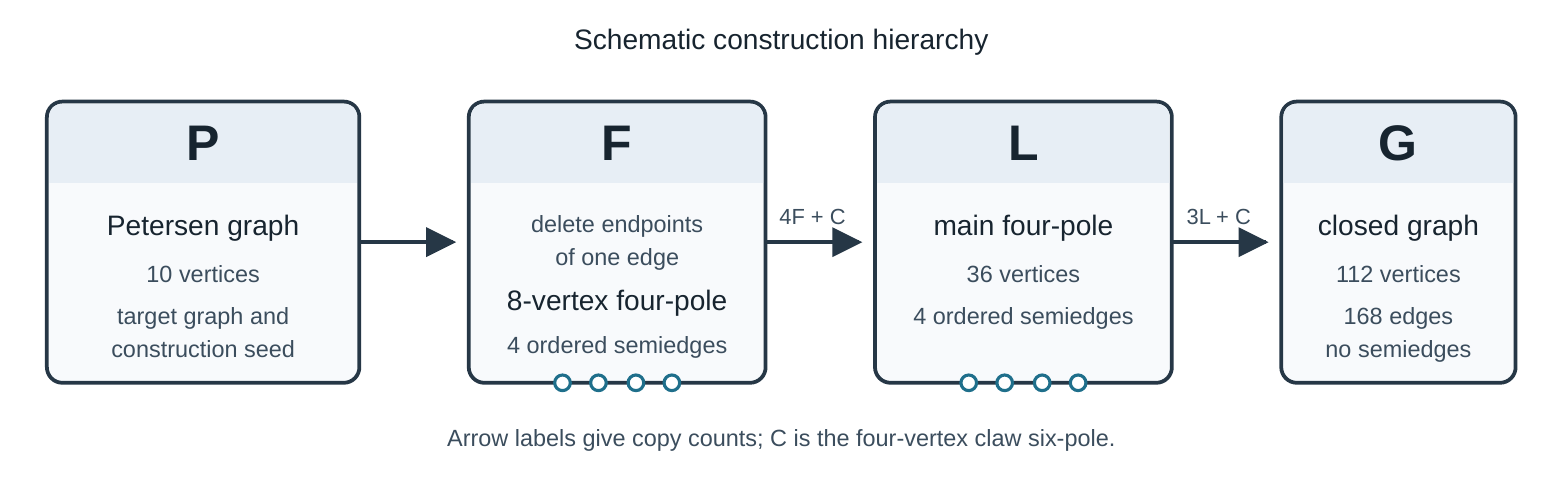}
\caption{Schematic hierarchy \(\Petersen\to F\to L\to G\).  The formulas
give copy counts; the semiedge joins are listed below.}
\label{fig:hierarchy}
\end{figure}

The semiedges of each multipole are ordered.  If \(M\) is a named copy, then
\(M.j\) denotes its semiedge with index \(j\).  Joining two semiedges produces
an edge between their incident vertices.

Let \(F\) be the four-pole obtained from the Petersen graph by deleting the
endpoints of one edge.  In the labeling used here, its internal edge set is
\[
E(F)=\bigl\{\{0,1\},\{0,5\},\{1,2\},\{1,6\},\{2,7\},
\{3,5\},\{3,6\},\{4,6\},\{4,7\},\{5,7\}\bigr\}.
\]
Its four semiedges are incident, in order, with vertices \((0,2,3,4)\).
The verification script checks this identification, including the stated
semiedge order, by an explicit bijection.

Let \(C\) be the claw six-pole
\[
 E(C)=\bigl\{\{0,1\},\{0,2\},\{0,3\}\bigr\}.
\]
Its six semiedges are incident, in order, with vertices
\((3,1,1,2,2,3)\).

To form \(L\), take four copies \(F_0,F_1,F_2,F_3\) and one copy of \(C\),
and make the nine joins in Table~\ref{tab:Ljoins}.  The four unjoined
semiedges, in order, are
\[
 (F_0.1,F_1.3,F_2.3,F_3.3).
\]

\begin{table}[ht]
\centering
\caption{The nine joins defining \(L=4F+C\).}
\label{tab:Ljoins}
\begin{tabular}{lll}
\toprule
\(F_0.0\!\sim\!F_3.2\) & \(F_0.2\!\sim\!C.0\) & \(F_0.3\!\sim\!C.1\)\\
\(F_1.1\!\sim\!C.2\) & \(F_1.0\!\sim\!C.3\) & \(F_2.0\!\sim\!C.4\)\\
\(F_3.0\!\sim\!C.5\) & \(F_1.2\!\sim\!F_2.1\) & \(F_2.2\!\sim\!F_3.1\)\\
\bottomrule
\end{tabular}
\end{table}

Finally take three copies \(L_0,L_1,L_2\) and another copy \(K\) of \(C\),
and join their remaining semiedges as in Table~\ref{tab:Gjoins}.  The resulting
graph has \(3\cdot36+4=112\) vertices.

\begin{figure}[ht]
\centering
\includegraphics[width=0.96\textwidth]{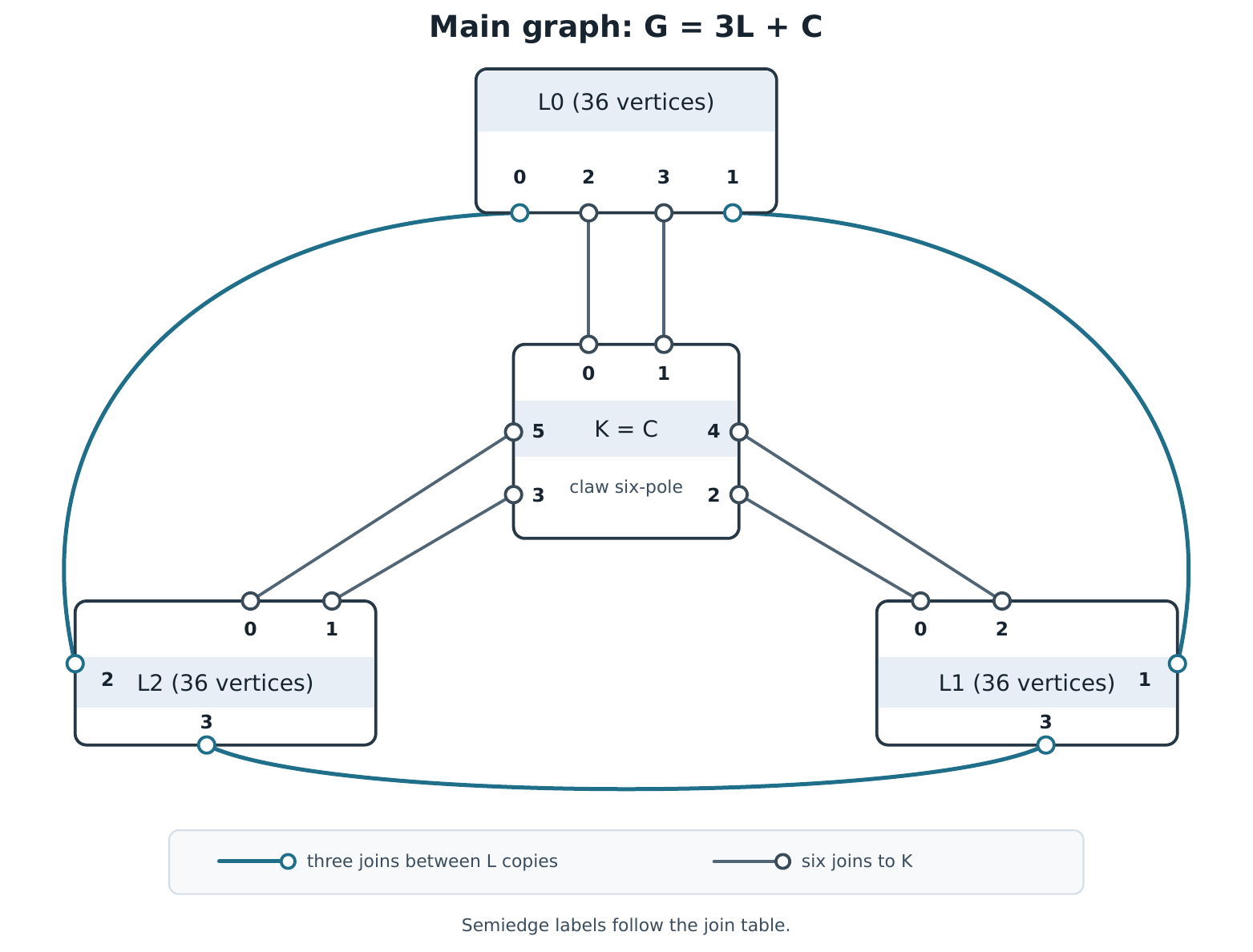}
\caption{Schematic macro-construction of \(G\); Table~\ref{tab:Gjoins}
specifies the nine joins.}
\label{fig:mainmacro}
\end{figure}

\begin{table}[ht]
\centering
\caption{The nine joins defining \(G=3L+C\).}
\label{tab:Gjoins}
\begin{tabular}{lll}
\toprule
\(L_0.0\!\sim\!L_2.2\) & \(L_0.1\!\sim\!L_1.1\) & \(L_0.2\!\sim\!K.0\)\\
\(L_0.3\!\sim\!K.1\) & \(L_1.0\!\sim\!K.2\) & \(L_1.2\!\sim\!K.4\)\\
\(L_2.1\!\sim\!K.3\) & \(L_2.0\!\sim\!K.5\) & \(L_1.3\!\sim\!L_2.3\)\\
\bottomrule
\end{tabular}
\end{table}

For the construction labeling, number the four \(F\)-copies and the \(C\)-copy
inside each \(L\) consecutively with offsets \(0,8,16,24,32\); use offsets
\(0,36,72\) for the three copies of \(L\), and offset 108 for \(K\).  The
following permutation \(\pi\) maps the construction labeling to the fixed
reference labeling:
\[
\pi(x)=
\begin{cases}
x,&0\le x\le7,\\
x+4,&8\le x\le31,\\
x-24,&32\le x\le35,\\
x+4,&36\le x\le43,\\
x+8,&44\le x\le67,\\
x-20,&68\le x\le71,\\
x+4,&72\le x\le79,\\
x+8,&80\le x\le103,\\
x-20,&104\le x\le107,\\
x-72,&108\le x\le111.
\end{cases}
\]

\begin{lemma}[Construction identity and structure]\label{lem:structure}
The preceding builder produces 168 distinct nonloop edges.  Applying \(\pi\)
gives exactly the edge set in Appendix~\ref{app:edges}, whose normalized digest
is the one in Theorem~\ref{thm:main}.  The graph is connected, cubic and
bridgeless, has girth five, and has edge- and vertex-connectivity three.
\end{lemma}

\begin{proof}
The script \code{verify\_main\_graph.py} reconstructs
\(KG(5,2)\), verifies the stated deletion model for \(F\), runs
the hierarchical builder, checks that \(\pi\) is a bijection, and compares the
mapped edge set edge by edge with the supplied reference edge set before hashing
it.  It then checks simplicity, degrees, connectivity and bridges.  Girth is
computed by breadth-first search from every root.  Finally, the script tests
every removal of fewer than three edges and every removal of fewer than three
vertices; none disconnects the graph.  Since the graph is cubic, its edge- and
vertex-connectivities are at most three, so both equal three.  The completed
run is recorded in \code{main\_graph\_verification.json}.
\end{proof}

\section{Two direct SAT encodings}\label{sec:sat}

Both encodings depend only on the explicit 168-edge list.  Before generating
either formula, the program checks the vertex set, simplicity, and degree
sequence.

\subsection{Direct Petersen encoding}

Using the notation in the archived variable maps, let \(P_{v,a}\) mean that
source vertex \(v\) chooses target vertex
\(a\in V(\Petersen)\), and let \(Q_{e,b}\) mean that source edge \(e\) maps to
target edge \(b\in E(\Petersen)\).  Standard exactly-one constraints select one
\(a\) for each \(v\) and one \(b\) for each \(e\).  Incidence clauses forbid
\(Q_{e,b}\) whenever \(e\) is incident with \(v\) but \(b\) is not incident
with the selected \(a\).  For each source star, pairwise clauses make its three
target-edge labels distinct.  Therefore they are exactly the star at the
selected target vertex.

The automorphism group of \(\Petersen\) induced by \(S_5\) is transitive on
ordered stars.  Using this symmetry, four unit clauses fix vertex 0 and its
deterministically ordered incident edges to one such star.

\subsection{Direct normal 5-edge-coloring encoding}

Let \(Y_{e,c}\) select one of five colors for source edge \(e\), and let
\(M_{v,A}\) select one of the ten 2-subsets \(A\subset[5]\) as the colors
missing at \(v\).  Standard exactly-one constraints select one \(c\) for each
\(e\) and one \(A\) for each \(v\).
Clauses make incident edge colors avoid \(A\), and properness clauses make the
three colors at a source vertex distinct; hence they are exactly
\([5]\setminus A\).  Across \(uv\), clauses allow missing pairs \(A,B\) only
when \(A=B\) or \(A\cap B=\varnothing\).  These are exactly the poor and rich
cases.  Using color-permutation symmetry, four unit clauses fix the ordered
colors at vertex 0.

\begin{lemma}[Encoding semantics]\label{lem:semantics}
The Petersen CNF is satisfiable exactly when the input graph has a Petersen
coloring.  The normal 5-edge-coloring CNF is satisfiable exactly when the input
graph has a normal 5-edge-coloring.  The four symmetry units in each
formula preserve satisfiability.
\end{lemma}

\begin{proof}
The preceding variable interpretations translate any satisfying assignment
to the stated coloring and conversely.  As a separate check,
\code{audit\_encoding\_semantics.py} exhausts every local target-star labeling,
every local palette labeling, every pair of endpoint missing pairs, and all
120 permutations of the five colors.  It also parses every raw and
symmetry-broken DIMACS file, checks the header counts and terminating zeros,
rejects out-of-range variables, repeated literals, and complementary literal
pairs within a clause, and compares the file hashes with the manifests.
\end{proof}

\section{Direct certificate results}\label{sec:certificates}

The certificate runs used CaDiCaL 3.0.1~\cite{BiereEtAl2024} at commit
\sha{c60730422e758ef1cebe7aeddf2dda31c996bf04} and \texttt{drat-trim} at
commit \sha{2e3b2dc0ecf938addbd779d42877b6ed69d9a985}~\cite{WetzlerEtAl2014}.
Table~\ref{tab:maincerts}
records the regenerated formula sizes and verification results.  Exact
hashes for the CNFs, DRAT certificates, maps, scripts, and logs are supplied
in the ancillary manifests.

\begin{center}
\refstepcounter{table}\label{tab:maincerts}
\textbf{Table \thetable: Direct SAT and certificate results for the main graph.}\par\smallskip
\footnotesize
\begin{tabular}{@{}lrrll@{}}
\toprule
Encoding & Variables & Clauses & Solver result & Proof check\\
\midrule
Petersen & 3640 & 68324 & \code{UNSATISFIABLE} & \code{VERIFIED}\\
Normal 5-edge-coloring & 1960 & 25484 & \code{UNSATISFIABLE} & \code{VERIFIED}\\
\bottomrule
\end{tabular}
\end{center}

CaDiCaL returned exit code 20 and the exact line
\code{s UNSATISFIABLE} for both formulas.
The checker returned exit code 0 and the exact normalized line
\code{s VERIFIED} for both proofs.

\begin{proof}[Proof of Theorem~\ref{thm:main}]
Lemma~\ref{lem:structure} identifies the explicit graph and establishes its
structural properties.  By Lemma~\ref{lem:semantics}, a Petersen coloring
would satisfy the Petersen CNF.  The checked DRAT derivation certifies that
this CNF is unsatisfiable, so no Petersen coloring exists.  The second checked
derivation separately establishes that the direct normal 5-edge-coloring CNF is
unsatisfiable.  This proves every assertion of the theorem.
\end{proof}

Following Ma, Mattiolo, Steffen, and Wolf~\cite{MaEtAl2025}, call a cubic
graph a \emph{3-graph} if every odd vertex set has at least three edges to its
complement.  A set of connected 3-graphs is \emph{3-complete} if every
connected 3-graph is colored by some member of the set; they denote the unique
inclusion-wise minimal such set by \(\mathcal H_3\).

\begin{corollary}[Infinitely many counterexamples]\label{cor:infinitely-many}
There are infinitely many connected simple bridgeless cubic graphs with no
Petersen coloring.  Moreover, the family \(\mathcal H_3\) is infinite.
\end{corollary}

\begin{proof}
A connected bridgeless cubic graph is a 3-graph: every cut of an odd vertex
set has odd size, and bridgelessness excludes a cut of size one.  Thus
Theorem~\ref{thm:main} disproves the Petersen Coloring Conjecture within the
class considered in \cite{MaEtAl2025}.  Their Theorem~4.8(i) says that, when
the conjecture is false, any set of connected 3-graphs that colors every
connected simple 3-graph is infinite.  If there were only finitely many
connected simple 3-graphs without a Petersen coloring, then those graphs,
together with \(\Petersen\), would form such a finite set: \(\Petersen\) colors
every noncounterexample, and each counterexample colors itself by the identity.
This is a contradiction, so there are infinitely many connected simple
3-graphs without a Petersen coloring.  Every connected 3-graph is bridgeless:
a bridge in a cubic graph separates an odd vertex set with a one-edge cut.

Their Theorem~1.1 identifies the conjecture with
\(\mathcal H_3=\{\Petersen\}\) and proves that otherwise \(\mathcal H_3\) is
infinite.
\end{proof}

\section{An additional nonisomorphic symmetric counterexample}\label{sec:d3}

This section gives a second counterexample \(H\), nonisomorphic to \(G\).

Let \(Q\) be the 36-vertex four-pole whose semiedges, labeled \(0,1,2,3\),
are incident respectively with vertices \(2,32,24,16\).  Its internal edge set is
\[
\begin{split}
E(Q)=\bigl\{&\{0,1\},\{0,5\},\{1,2\},\{1,6\},\{2,7\},\{3,5\},\{3,6\},\\
&\{4,6\},\{4,7\},\{5,7\},\{8,9\},\{8,10\},\{8,11\},\{9,14\},\\
&\{10,12\},\{11,28\},\{12,13\},\{12,17\},\{13,14\},\{13,18\},\{14,19\},\\
&\{15,17\},\{15,18\},\{15,22\},\{16,18\},\{16,19\},\{17,19\},\{20,21\},\\
&\{20,25\},\{21,22\},\{21,26\},\{22,27\},\{23,25\},\{23,26\},\{23,30\},\\
&\{24,26\},\{24,27\},\{25,27\},\{28,29\},\{28,33\},\{29,30\},\{29,34\},\\
&\{30,35\},\{31,33\},\{31,34\},\{32,34\},\{32,35\},\{33,35\},\{0,31\},\\
&\{3,11\},\{4,9\},\{10,20\}\bigr\}.
\end{split}
\]

Take copies \(Q_0,Q_1,Q_2\), indices modulo three, a center \(c\), and leaves
\(\ell_0,\ell_1,\ell_2\).  For every \(i\), join
\[
 Q_i.0\sim Q_{i+1}.3,\qquad
 Q_i.2\sim\ell_i\sim Q_{i+1}.1,\qquad
 c\sim\ell_i.
\]

\begin{table}[ht]
\centering
\caption{The twelve joins defining the additional graph \(H\).}
\label{tab:Hjoins}
\small
\begin{tabular}{llll}
\toprule
Join family & \(i=0\) & \(i=1\) & \(i=2\)\\
\midrule
Direct cycle
  & \(Q_0.0\!\sim\!Q_1.3\)
  & \(Q_1.0\!\sim\!Q_2.3\)
  & \(Q_2.0\!\sim\!Q_0.3\)\\
Leaf exit
  & \(Q_0.2\!\sim\!\ell_0\)
  & \(Q_1.2\!\sim\!\ell_1\)
  & \(Q_2.2\!\sim\!\ell_2\)\\
Leaf entry
  & \(\ell_0\!\sim\!Q_1.1\)
  & \(\ell_1\!\sim\!Q_2.1\)
  & \(\ell_2\!\sim\!Q_0.1\)\\
Central claw
  & \(c\!\sim\!\ell_0\)
  & \(c\!\sim\!\ell_1\)
  & \(c\!\sim\!\ell_2\)\\
\bottomrule
\end{tabular}
\end{table}

In the supplied labeling, \(Q_i\) occupies \(36i,\ldots,36i+35\),
\(c=108\), and \(\ell_i=109+i\).

\begin{figure}[ht]
\centering
\includegraphics[width=0.94\textwidth]{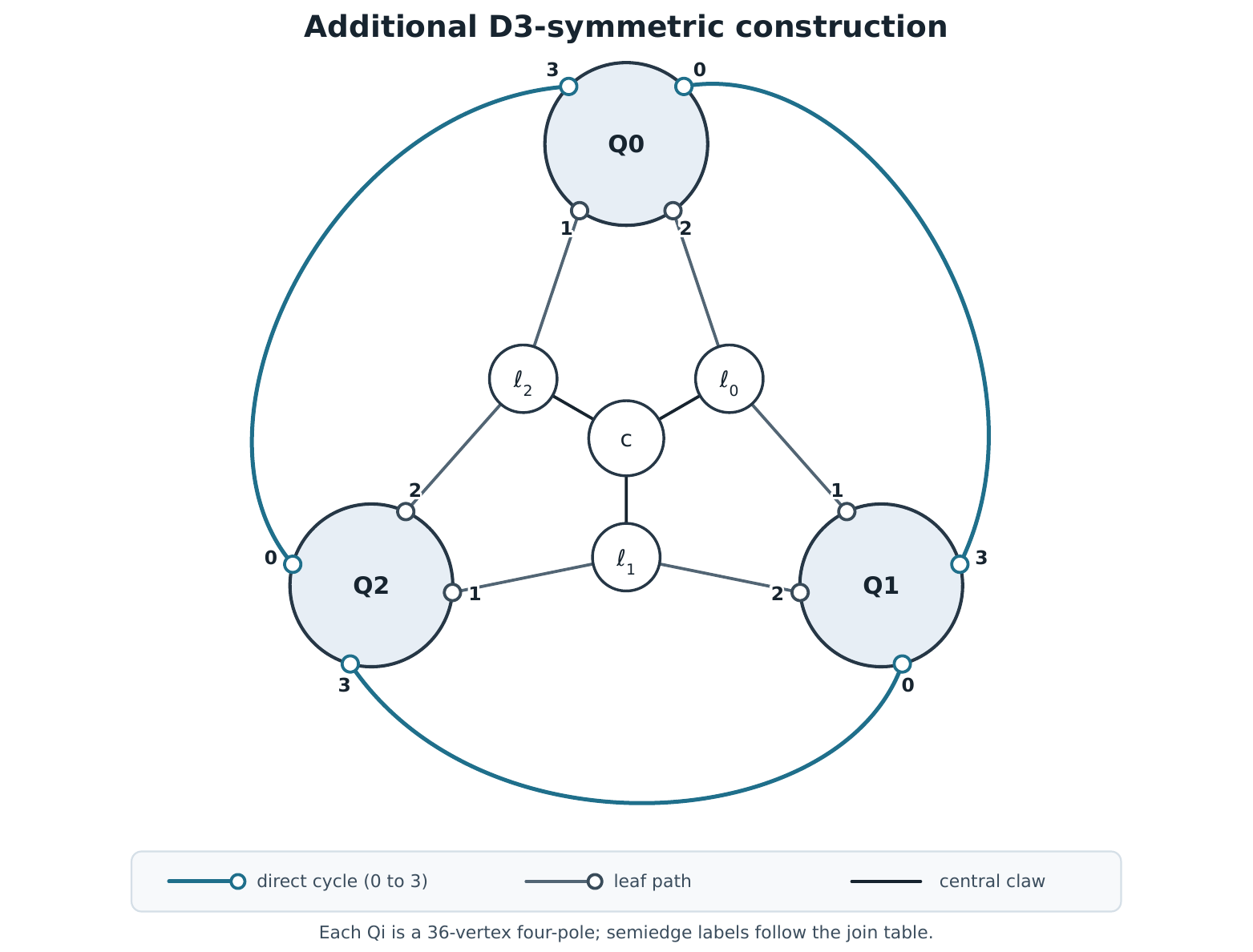}
\caption{Schematic three-copy construction of \(H\); Table~\ref{tab:Hjoins}
specifies the twelve joins.}
\label{fig:d3}
\end{figure}

\begin{theorem}[Alternative counterexample]\label{thm:d3}
The graph \(H\) is a simple connected bridgeless cubic graph on 112 vertices
and 168 edges, with girth five; its edge-connectivity and vertex-connectivity
are both three.  It has no
Petersen coloring and no normal 5-edge-coloring.  Its full automorphism
group is \(D_3\) of order six, and it is nonisomorphic to the graph in
Theorem~\ref{thm:main}.  The SHA-256 digest of its normalized sorted edge list is
\[
\sha{0f2d8858110c6f012de7ddffa92fdbc709d7da630f199b0e3c81bb56eb6b35c7}.
\]
\end{theorem}

\begin{proof}[Computational verification]\leavevmode\par\noindent
The verification program reconstructs the exact edge list from the displayed
rule, checks its graph properties, and computes the full automorphism groups
of \(H\) and the main graph, obtaining orders six and one.  Both direct CNFs of
Section~\ref{sec:sat} were generated from the complete edge list of \(H\).
CaDiCaL returned UNSAT for both, and \texttt{drat-trim} verified both DRAT
proofs.  The results are given in Table~\ref{tab:d3certs}.
\end{proof}

\begin{center}
\refstepcounter{table}\label{tab:d3certs}
\textbf{Table \thetable: Direct SAT and certificate results for \(H\).}\par\smallskip
\footnotesize
\begin{tabular}{@{}lrrll@{}}
\toprule
Encoding & Variables & Clauses & Solver result & Proof check\\
\midrule
Petersen & 3640 & 68324 & \code{UNSATISFIABLE} & \code{VERIFIED}\\
Normal 5-edge-coloring & 1960 & 25484 & \code{UNSATISFIABLE} & \code{VERIFIED}\\
\bottomrule
\end{tabular}
\end{center}

\section{Additional verification controls}\label{sec:audit}

The certificate runner accepts an UNSAT result only when CaDiCaL exits with
code 20 and prints the whole line \code{s UNSATISFIABLE}; it accepts a proof
only when \texttt{drat-trim} exits with code 0 and prints the whole line
\code{s VERIFIED}.  A supplied negative control exits successfully while
printing \code{s NOT VERIFIED}, and the runner rejects it.

As a satisfiable control, the same generators encode the Petersen graph.
CaDiCaL returned SAT in both formulations, and the returned assignments were
checked against every generated clause.  The Petersen-coloring and normal
5-edge-coloring encodings use separate variables and local constraints, and
each graph--encoding pair has its own CNF and checked DRAT proof.

The ancillary documentation lists the relevant files and reproduction steps,
and the compact archive has its own SHA-256 manifest.

\section*{Computational provenance and responsibility}

OpenAI language-model systems were used extensively in the discovery,
computational search, verification, and preparation of this work.  The author
reviewed the final claims and artifacts and accepts responsibility for the
contents.

\clearpage
\section*{Artifact availability}

The frozen version-1.1.0 Zenodo record corresponding to this manuscript is
available at the immutable version DOI
\href{https://doi.org/10.5281/zenodo.21845291}{10.5281/zenodo.21845291}
\cite{PutmanZenodo2026}.  It contains the manuscript and source files, a compact
ancillary archive with the exact graphs, construction scripts, SAT encodings,
verification software, completed logs, and manifests, and a separate archive
containing the four DRAT proof certificates.  Reproduction instructions and
exact SHA-256 digests for all artifacts are provided in the accompanying
documentation.

\clearpage
\appendix
\section{Complete labeled edge list}\label{app:edges}

The vertex set is \(\{0,1,\ldots,111\}\).  Each unordered pair below is one
edge; the list contains 168 distinct pairs.

\begin{center}
\scriptsize
\renewcommand{\arraystretch}{0.96}
\begin{longtable}{@{}lll@{}}
\toprule
Edges 1--56 & Edges 57--112 & Edges 113--168\\
\midrule
\endhead
\bottomrule
\endlastfoot
$0\!\!-\!\!1$ & $36\!\!-\!\!37$ & $72\!\!-\!\!74$ \\
$0\!\!-\!\!5$ & $36\!\!-\!\!38$ & $72\!\!-\!\!75$ \\
$0\!\!-\!\!31$ & $36\!\!-\!\!39$ & $72\!\!-\!\!108$ \\
$1\!\!-\!\!2$ & $37\!\!-\!\!42$ & $73\!\!-\!\!75$ \\
$1\!\!-\!\!6$ & $38\!\!-\!\!64$ & $76\!\!-\!\!77$ \\
$2\!\!-\!\!7$ & $38\!\!-\!\!92$ & $76\!\!-\!\!81$ \\
$2\!\!-\!\!100$ & $39\!\!-\!\!78$ & $76\!\!-\!\!107$ \\
$3\!\!-\!\!5$ & $40\!\!-\!\!41$ & $77\!\!-\!\!78$ \\
$3\!\!-\!\!6$ & $40\!\!-\!\!45$ & $77\!\!-\!\!82$ \\
$3\!\!-\!\!11$ & $40\!\!-\!\!71$ & $78\!\!-\!\!83$ \\
$4\!\!-\!\!6$ & $41\!\!-\!\!42$ & $79\!\!-\!\!81$ \\
$4\!\!-\!\!7$ & $41\!\!-\!\!46$ & $79\!\!-\!\!82$ \\
$4\!\!-\!\!9$ & $42\!\!-\!\!47$ & $79\!\!-\!\!87$ \\
$5\!\!-\!\!7$ & $43\!\!-\!\!45$ & $80\!\!-\!\!82$ \\
$8\!\!-\!\!9$ & $43\!\!-\!\!46$ & $80\!\!-\!\!83$ \\
$8\!\!-\!\!10$ & $43\!\!-\!\!51$ & $80\!\!-\!\!85$ \\
$8\!\!-\!\!11$ & $44\!\!-\!\!46$ & $81\!\!-\!\!83$ \\
$9\!\!-\!\!14$ & $44\!\!-\!\!47$ & $84\!\!-\!\!85$ \\
$10\!\!-\!\!12$ & $44\!\!-\!\!49$ & $84\!\!-\!\!86$ \\
$10\!\!-\!\!20$ & $45\!\!-\!\!47$ & $84\!\!-\!\!87$ \\
$11\!\!-\!\!28$ & $48\!\!-\!\!49$ & $85\!\!-\!\!90$ \\
$12\!\!-\!\!13$ & $48\!\!-\!\!50$ & $86\!\!-\!\!88$ \\
$12\!\!-\!\!17$ & $48\!\!-\!\!51$ & $86\!\!-\!\!96$ \\
$13\!\!-\!\!14$ & $49\!\!-\!\!54$ & $87\!\!-\!\!104$ \\
$13\!\!-\!\!18$ & $50\!\!-\!\!52$ & $88\!\!-\!\!89$ \\
$14\!\!-\!\!19$ & $50\!\!-\!\!60$ & $88\!\!-\!\!93$ \\
$15\!\!-\!\!17$ & $51\!\!-\!\!68$ & $89\!\!-\!\!90$ \\
$15\!\!-\!\!18$ & $52\!\!-\!\!53$ & $89\!\!-\!\!94$ \\
$15\!\!-\!\!22$ & $52\!\!-\!\!57$ & $90\!\!-\!\!95$ \\
$16\!\!-\!\!18$ & $53\!\!-\!\!54$ & $91\!\!-\!\!93$ \\
$16\!\!-\!\!19$ & $53\!\!-\!\!58$ & $91\!\!-\!\!94$ \\
$16\!\!-\!\!56$ & $54\!\!-\!\!59$ & $91\!\!-\!\!98$ \\
$17\!\!-\!\!19$ & $55\!\!-\!\!57$ & $92\!\!-\!\!94$ \\
$20\!\!-\!\!21$ & $55\!\!-\!\!58$ & $92\!\!-\!\!95$ \\
$20\!\!-\!\!25$ & $55\!\!-\!\!62$ & $93\!\!-\!\!95$ \\
$21\!\!-\!\!22$ & $56\!\!-\!\!58$ & $96\!\!-\!\!97$ \\
$21\!\!-\!\!26$ & $56\!\!-\!\!59$ & $96\!\!-\!\!101$ \\
$22\!\!-\!\!27$ & $57\!\!-\!\!59$ & $97\!\!-\!\!98$ \\
$23\!\!-\!\!25$ & $60\!\!-\!\!61$ & $97\!\!-\!\!102$ \\
$23\!\!-\!\!26$ & $60\!\!-\!\!65$ & $98\!\!-\!\!103$ \\
$23\!\!-\!\!30$ & $61\!\!-\!\!62$ & $99\!\!-\!\!101$ \\
$24\!\!-\!\!26$ & $61\!\!-\!\!66$ & $99\!\!-\!\!102$ \\
$24\!\!-\!\!27$ & $62\!\!-\!\!67$ & $99\!\!-\!\!106$ \\
$24\!\!-\!\!39$ & $63\!\!-\!\!65$ & $100\!\!-\!\!102$ \\
$25\!\!-\!\!27$ & $63\!\!-\!\!66$ & $100\!\!-\!\!103$ \\
$28\!\!-\!\!29$ & $63\!\!-\!\!70$ & $101\!\!-\!\!103$ \\
$28\!\!-\!\!33$ & $64\!\!-\!\!66$ & $104\!\!-\!\!105$ \\
$29\!\!-\!\!30$ & $64\!\!-\!\!67$ & $104\!\!-\!\!109$ \\
$29\!\!-\!\!34$ & $65\!\!-\!\!67$ & $105\!\!-\!\!106$ \\
$30\!\!-\!\!35$ & $68\!\!-\!\!69$ & $105\!\!-\!\!110$ \\
$31\!\!-\!\!33$ & $68\!\!-\!\!73$ & $106\!\!-\!\!111$ \\
$31\!\!-\!\!34$ & $69\!\!-\!\!70$ & $107\!\!-\!\!109$ \\
$32\!\!-\!\!34$ & $69\!\!-\!\!74$ & $107\!\!-\!\!110$ \\
$32\!\!-\!\!35$ & $70\!\!-\!\!75$ & $108\!\!-\!\!110$ \\
$32\!\!-\!\!37$ & $71\!\!-\!\!73$ & $108\!\!-\!\!111$ \\
$33\!\!-\!\!35$ & $71\!\!-\!\!74$ & $109\!\!-\!\!111$ \\

\end{longtable}
\end{center}

\section{Exact target numbering}\label{app:numbering}

For the Petersen encoding, target vertices are the lexicographically ordered
pairs
\[
(0,1),(0,2),(0,3),(0,4),(1,2),(1,3),(1,4),(2,3),(2,4),(3,4),
\]
and target edges are the lexicographically ordered pairs of disjoint target
vertices.  Source edges are the lexicographically ordered pairs in
Appendix~\ref{app:edges}.  The 120 target permutations are induced by all
permutations of \(\{0,1,2,3,4\}\).  The normal 5-edge-coloring encoding orders
its ten missing pairs lexicographically.  The archived JSON variable maps
specify these conventions exactly.

\end{document}